%% file: doc.tex
\documentclass[envcountsame]{llncs}

\usepackage[utf8]{inputenc}
\usepackage[charter]{mathdesign}
\usepackage{xspace}
\usepackage{enumitem}
\usepackage{wrapfig}
\usepackage{todonotes}
\usepackage{hyperref}

\makeatletter
\newcommand\footnoteref[1]{\protected@xdef\@thefnmark{\ref{#1}}\@footnotemark}
\makeatother

\hypersetup{
	pdfborder={0 0 0}
}

\definecolor{grey}{rgb}{0.95,0.95,0.95}
\definecolor{green}{rgb}{0.2,0.6,0.4}

\newcommand{\imp}{\rightarrow}

\newcommand{\Pb}{\mathbb{P}}

\newcommand{\Psf}{\mathsf{P}}
\newcommand{\Qsf}{\mathsf{Q}}

\newcommand{\Ccal}{\mathcal{C}}

\newcommand{\Ical}{\mathcal{I}}

\newcommand{\Mcal}{\mathcal{M}}

\newcommand{\Qcal}{\mathcal{Q}}
\newcommand{\Rcal}{\mathcal{R}}
\newcommand{\Scal}{\mathcal{S}}
\newcommand{\Kcal}{\mathcal{K}}

\renewcommand{\setminus}{\smallsetminus}

\newcommand{\subseteqfin}{\subseteq_{\tiny \texttt{fin}}}

\newcommand{\tuple}[1]{\left\langle #1 \right\rangle}

\newcommand{\s}[1]{\ensuremath{\sf{#1}}}

\newcommand{\rca}{\s{RCA}_0}
\newcommand{\aca}{\s{ACA}_0}
\newcommand{\wkl}{\s{WKL}_0}

\newcommand{\kl}{\s{KL}}
\newcommand{\rt}{\s{RT}}

\newcommand{\rrt}{\s{RRT}}

\newcommand{\ads}{\s{ADS}}
\newcommand{\sads}{\s{SADS}}

\newcommand{\scac}{\s{SCAC}}

\newcommand{\coh}{\s{COH}}
\newcommand{\pizog}{\Pi^0_1\s{G}}

\renewcommand{\ts}{\s{TS}}
\newcommand{\sts}{\s{STS}}

\newcommand{\fs}{\s{FS}}

\newcommand{\emo}{\s{EM}}

\newcommand{\amt}{\s{AMT}}

\usepackage{xcolor}	
\usepackage{soul}

\title{Iterative forcing and hyperimmunity\\ in reverse mathematics}
\author{
  Ludovic Patey
}

\institute{
Laboratoire PPS, Universit\'e Paris Diderot, Paris, FRANCE\\
\email{ludovic.patey@computability.fr}
}

\date{\today}

\begin{document}

\maketitle

\begin{abstract}
The separation between two theorems in reverse mathematics
is usually done by constructing a Turing ideal satisfying a theorem P
and avoiding the solutions to a fixed instance of a theorem Q.
Lerman, Solomon and Towsner introduced a forcing technique
for iterating a computable non-reducibility in order to separate
theorems over omega-models. In this paper, we present a modularized
version of their framework in terms of preservation of hyperimmunity
and show that it is powerful enough
to obtain the same separations results as Wang did with
his notion of preservation of definitions.
\end{abstract}

\section{Introduction}
\input{parts/part0-introduction}

\section{The iteration framework}\label{sect:framework}
\input{parts/part1-framework}

\section{Preservation of hyperimmunity}\label{sect:hyperimmunity}
\input{parts/part2-hyperimmunity}

\section{Basic preservations of hyperimmunity}\label{sect:basic-hyperimmunity}
\input{parts/part3-basic-hyperimmunity}

\section{The Erd\H{o}s-Moser theorem and preservation of hyperimmunity}\label{sect:em-hyperimmunity}
\input{parts/part4-em-hyperimmunity}

\section{Thin set theorem and preservation of hyperimmunity}\label{sect:ts-hyperimmunity}
\input{parts/part5-ts2-hyperimmunity}

\vspace{0.5cm}

\noindent \textbf{Acknowledgements}. 
The author is thankful to Wei Wang for useful comments and discussions.

\vspace{0.5cm}

\bibliographystyle{plain}

\clearpage
\appendix

\input{parts/part6-appendix}

\end{document}

%% file: parts/part0-introduction.tex
Reverse mathematics is a mathematical program which aims to capture the provability
content of ordinary (i.e. non set-theoretic) theorems. It uses the framework of subsystems of second-order
arithmetic, with a base theory~$\rca$ which is composed of the basic axioms of Peano arithmetic
together with the~$\Delta^0_1$ comprehension scheme and the~$\Sigma^0_1$ induction scheme.
Thanks to the equivalence between~$\Delta^0_1$-definable sets and computable sets, 
$\rca$ can be thought as capturing ``computational mathematics''.
See~\cite{Hirschfeldt2013Slicing} for a good introduction.

Many theorems are~$\Pi^1_2$ statements~$(\forall X)(\exists Y)\Phi(X,Y)$
and come with a natural class of~\emph{instances}~$X$. The sets~$Y$
such that~$\Phi(X,Y)$ holds are~\emph{solutions} to~$X$.
For example, König's lemma ($\kl$) states that every infinite, finitely branching
tree has an infinite path. An instance of~$\kl$ is an infinite, finitely branching tree~$T$.
A solution to~$T$ is an infinite path through~$T$.
Given two~$\Pi^1_2$ statements~$\Psf$ and~$\Qsf$, proving an implication~$\Qsf \imp \Psf$
over~$\rca$ consists in taking a~$\Psf$-instance $X$ and constructing a solution to~$X$
through a computational process involving several applications of the~$\Qsf$ statement.
Empirically, many proofs of implications are in fact~\emph{computable reductions}~\cite{Hirschfeldtnotions}.

\begin{definition}[Computable reducibility]
Fix two~$\Pi^1_2$ statements~$\Psf$ and~$\Qsf$.
We say that $\Psf$ is \emph{computably reducible} to~$\Qsf$ (written $\Psf \leq_c \Qsf$)
if every~$\Psf$-instance~$I$ computes a~$\Qsf$-instance~$J$ such that for every solution~$X$ to~$J$,
$X \oplus I$ computes a solution to~$I$.
\end{definition}

If the computable reduction between from~$\Psf$ to~$\Qsf$ can be formalized over~$\rca$, then
$\rca \vdash \Qsf \imp \Psf$. However, $\Psf$ may not be computably reducible
to~$\Qsf$ while~$\rca \vdash \Qsf \imp \Psf$. Indeed, one may need more than one application
of~$\Qsf$ to solve the instance of~$\Psf$. This is for example the case of Ramsey's theorem for
pairs with~$n$ colors ($\rt^2_n$) which implies~$\rt^2_{n+1}$ over~$\rca$, but~$\rt^2_{n+1} \not \leq_c \rt^2_n$
for~$n \geq 1$ (see~\cite{Patey2015weakness}).

In order to prove the non-implication between~$\Psf$ and~$\Qsf$, one needs to iterate
the computable non-reducibility in order to build a model of~$\Qsf$ which is not a model of~$\Psf$.
This is the purpose of the framework developed by Lerman, Solomon and Towsner in~\cite{Lerman2013Separating}.
They successfully used their framework for separating the Erd\H{o}s-Moser theorem ($\emo$)
from the stable ascending descending sequence principle~($\sads$)
and separating the ascending descending sequence~($\ads$) from the stable chain
antichain principle~($\scac$). Their approach has been reused by Flood \& Towsner~\cite{Flood2014Separating}
and the author~\cite{Patey2015Ramsey} on diagonal non-computability statements.

However, their framework suffers some drawbacks. In particular
the forcing notions involved are heavy and
the deep combinatorics witnessing the non-implications are hidden by the complexity of the proof.
As well, the~$\Psf$-instance chosen in the ground forcing depends on the forcing notion
used in the iteration forcing and therefore the overall construction is not modular.
On the other side, Wang~\cite{Wang2014Definability} recently
introduced the notion of preservation of definitions 
and made independent proofs of preservations
for various statements included~$\emo$. Then he deduced that the conjunction
of those statements does not imply~$\sads$, therefore strengthening
the result of Lerman, Solomon \& Towsner in a modular way.
Variants of this notion have been reused by the author~\cite{Patey2015weakness} for separating the free set theorem~($\fs$)
from~$\rt^2_2$.

In this paper, we present a modularized version of the framework
of Lerman, Solomon \& Towsner and use it to reprove the separation results
obtained by Wang~\cite{Wang2014Definability}. We thereby show that this framework
is a viable alternative to the notion introduced by Wang
for separating statements in reverse mathematics. In particular, we reprove the following theorem,
in which $\coh$ is the cohesiveness principle, $\wkl$ is the weak König's lemma,
$\rrt^2_2$ the rainbow Ramsey theorem for pairs, $\pizog$ the~$\Pi^0_1$-genericity principle
and~$\sts^2$ the stable thin set theorem for pairs.

\begin{theorem}[Wang~\cite{Wang2014Definability}]\label{thm:wang-theorem}
Let~$\Phi$ be the conjunction of~$\coh$, $\wkl$, $\rrt^2_2$,
$\pizog$, and~$\emo$. Over~$\rca$, $\Phi$
does not imply any of~$\sads$ and $\sts^2$.
\end{theorem}

In section~\ref{sect:framework}, we introduce the framework of Lerman, Solomon \& Towsner
in its original form and detail its drawbacks. Then, in section~\ref{sect:hyperimmunity},
we develop a modularized version of their framework. In section~\ref{sect:basic-hyperimmunity},
we establish basic preservation results, before reproving in section~\ref{sect:em-hyperimmunity}
Wang's theorem. Last, we reprove in section~\ref{sect:ts-hyperimmunity} 
the separation obtained by the author in~\cite{Patey2015weakness}.

\subsection{Notations}

\emph{String, sequence}.
Fix an integer $k \in \omega$.
A \emph{string} (over $k$) is an ordered tuple of integers $a_0, \dots, a_{n-1}$
(such that $a_i < k$ for every $i < n$). The empty string is written~$\varepsilon$. A \emph{sequence}  (over $k$)
is an infinite listing of integers $a_0, a_1, \dots$ (such that $a_i < k$ for every $i \in \omega$).
Given $s \in \omega$,
$k^s$ is the set of strings of length $s$ over~$k$. As well,
$k^{<\omega}$ is the set of finite strings over~$k$. 
Given a string $\sigma \in k^{<\omega}$, we denote by $|\sigma|$ its length.
Given two strings $\sigma, \tau \in k^{<\omega}$, $\sigma$ is a \emph{prefix}
of $\tau$ (written $\sigma \preceq \tau$) if there exists a string $\rho \in k^{<\omega}$
such that $\sigma \rho = \tau$.
A \emph{binary string} (resp. real) is a \emph{string} (resp. sequence) over $2$.
We may equate a real with a set of integers by considering that the real is its characteristic function.

\emph{Tree, path}.
A tree $T \subseteq \omega^{<\omega}$ is a set downward closed under the prefix relation.
The tree~$T$ is \emph{finitely branching} if every node~$\sigma \in T$
has finitely many immediate successors.
A \emph{binary} tree is a tree~$T \subseteq 2^{<\omega}$.
A set $P \subseteq \omega$ is a \emph{path} though~$T$ if for every $\sigma \prec P$,
$\sigma \in T$. A string $\sigma \in k^{<\omega}$ is a \emph{stem} of a tree $T$
if every $\tau \in T$ is comparable with~$\sigma$.
Given a tree $T$ and a string $\sigma \in T$,
we denote by $T^{[\sigma]}$ the subtree $\{\tau \in T : \tau \preceq \sigma \vee \tau \succeq \sigma\}$.

\emph{Sets}. Given two sets~$X$ and~$Y$, $X \subseteq^{*} Y$ means that~$X$ is almost included into~$Y$,
$X =^{*} Y$ means~$X \subseteq^{*} Y \wedge Y \subseteq^{*} X$ and $X \subseteqfin Y$ means that~$X$
is a finite subset of~$Y$. Given some~$x \in \omega$, $A > x$ denotes the formula~$(\forall y \in A)[y > x]$.

%% file: parts/part1-framework.tex
An \emph{$\omega$-structure} is a structure $\Mcal = (\omega, S, +, \cdot, <)$
where $\omega$ is the set of standard integers, $+$, $\cdot$ and $<$ 
are the standard operations over integers and $S$ is a set of reals
such that~$\Mcal$ satisfies the axioms of~$\rca$.
Friedman~\cite{Friedman1974Some} characterized the second-order parts~$S$ of~$\omega$-structures
as those forming a \emph{Turing ideal}, that is, a set of reals closed under Turing join
and downward-closed under Turing reduction.

Fix two $\Pi^1_2$ statements~$\Psf$ and $\Qsf$.
The construction of an~$\omega$-model of~$\Psf$ which is not a model of~$\Qsf$
consists in creating a Turing ideal~$\Ical$ together with a fixed $\Qsf$-instance $I \in \Ical$,
such that every $\Psf$-instance $J \in \Ical$ has a solution in~$\Ical$,
whereas $I$ contains no solution in~$\Ical$.
In the first place, let us just focus on the one-step case, that is, a proof that~$\Qsf \not \leq_c \Psf$.
To do so, one has to choose carefully some~$\Qsf$-instance~$I$ such that
every $I$-computable $\Psf$-instance has a solution $X$ which does not~$I$-compute
a solution to~$I$.
The construction of a solution~$X$ to some~$I$-computable $\Psf$-instance~$J$ will have to satisfy
the following scheme of requirements for each index~$e$:
\[
\Rcal_e : \Phi^{X \oplus I}_e \mbox{ infinite } \imp \Phi^{X \oplus I}_e \mbox{ is not a solution to } I
\]
Such requirements may not be satisfiable for an arbitrary $\Qsf$-instance~$I$.
The choice of the instance and the satisfaction of the requirement is strongly dependent
on the combinatorics of the statement~$\Qsf$ and the forcing notion used for constructing
a solution to~$J$. A recurrent approach in the framework
of Lerman, Solomon \& Towsner consists in constructing a $\Qsf$-instance~$I$
which satisfies some fairness property.
The forcing notion~$\Pb^I$ used in the construction of a solution to~$J$
is usually designed so that
\begin{itemize}
	\item[(i)] There exists an $I$-computable set encoding (at least) every condition in~$\Pb^I$
	\item[(ii)] Given some forcing condition in~$\Pb^I$, one can
	uniformly find in a c.e.\ search a finite set of candidate extensions
	such that one of them is in~$\Pb^I$ (e.g.\ the notion of split pair in~\cite{Lerman2013Separating},
of finite cover for a tree forcing, ...).
\end{itemize}
The fairness property states the following:

\smallskip
{\itshape For every condition in~$\Pb^I$,
if for every~$x \in \omega$, there exists a \emph{finite} $\Qsf$-instance~$A > x$
and a finite set of candidate extensions~$d_0, \dots, d_m$ such that~$\Phi^{d_i \oplus I}_e$ is not a
solution to~$A$ for each~$i \leq m$, then one of the~$A$'s is a subinstance of~$I$.}
\smallskip

This property is designed so that we can satisfy it by taking 
each condition~$c \in \Pb^I$ one at a time, find some finite $\Qsf$-instance $A$ on which~$I$
is not yet defined, and define~$I$ over~$A$.
One can think of the instance~$I$ as a fair adversary who, if we have infinitely often
the occasion to beat him, will be actually beaten at some time.

Suppose now we want to extend this computable non-reducibility into
a separation over~$\omega$-structures. One may naturally try to make
the instance~$I$ satisfy the fairness property at every level
of the iteration forcing. At the first iteration with an~$I$-computable~$\Psf$-instance~$J$, the property is unchanged.
At the second iteration, the~$\Psf$-instance~$J_1$ is~$X_0 \oplus I$-computable,
but the set~$X_0$ is not yet constructed. Hopefully, the fairness property requires a finite piece of oracle~$X_0$.
Therefore we can modify the fairness property which becomes

\smallskip
{\itshape For every condition~$c_0 \in \Pb^I$ and every condition~$c_1 \in \Pb^{c_0 \oplus I}$,
if for every~$x \in \omega$, there exists a $\Qsf$-instance~$A > x$,
a finite set of candidate extensions~$d_0, \dots, d_m \in \Pb^I$
and~$d_{0,i}, \dots, d_{n_i,i} \in \Pb^{d_i \oplus I}$ for each~$i \leq m$ such that
$\Phi^{d_{j,i} \oplus d_i \oplus I}_e$ is not a solution to~$A$ 
for each~$i \leq m$ and~$j \leq n_i$, then one of the~$A$'s is a subinstance of~$I$.}
\smallskip

Since this property becomes overly complicated in the general case, Lerman, Solomon and Towsner
abstracted the notion of requirement and made it a $\Sigma^{0,I}_1$ black box
which takes as parameters a condition and a finite~$\Qsf$-instance.
Instead of making the instance~$I$ in charge of satisfying the fairness property at every level
of the iteration forcing, the instance~$I$ satisfies the property only at the first level. Then, by encoding a requirement
at the next level into a requirement at the current level, the iteration forcing
ensures the propagation of this fairness property from the first level to every level.
The property in its abstracted form is then

\smallskip
{\itshape For every condition in~$\Pb^I$ and every $\Sigma^{0,I}_1$ predicate~$\Kcal^I$,
if for every~$x \in \omega$, there exists a \emph{finite} $\Qsf$-instance~$A > x$
and a finite set of candidate extensions~$d_0, \dots, d_m$ such that~$\Kcal^I(A, d_i)$ is satisfied for each~$i \leq m$, 
then one of the~$A$'s is a subinstance of~$I$.}
\smallskip

In particular, by letting~$\Kcal^I(A, c)$ be the predicate ``$\Phi^{d_i \oplus I}_e$
is not a solution to~$A$'', the requirements~$\Rcal_e$ will be satisfied.

The problem of such an approach is that the construction of the $\Qsf$-instance
strongly depends on the forcing notion used in the iteration forcing.
A slight modification of the latter requires to modify the ground forcing.
As well, if someone wants to prove that the conjunction of two statements
does not imply a third one, we need to construct an instance~$I$ which will
satisfy the fairness property for the two statements, and in each iteration forcing,
we will need to ensure that both properties are propagated to the next iteration.
The size of the overall construction explodes when trying to make a separation of the conjunction
of several statements at the same time.

%% file: parts/part2-hyperimmunity.tex
In this section, we propose a general simplification
of the framework of Lerman, Solomon \& Towsner~\cite{Lerman2013Separating}
and illustrate it in the case of the separation of~$\emo$ from~$\sads$.
The corresponding fairness property happens to coincide with the notion of hyperimmunity.
The underlying idea ruling this simplification is the following:
since each condition in the iteration forcing can be given an index
and since the finite set of candidate extensions of a condition~$c$, 
can be found in a c.e.\ search, given a $\Sigma^{0,I}_1$ predicate~$\Kcal^I$,
the following formula is again $\Sigma^{0,I}_1$:

\smallskip
{\itshape $\varphi(U) = $ ``there exists a finite set of candidate extensions
$d_0, \dots, d_m$ of~$c$ such that $\Kcal^I(U, d_i)$ is satisfied for each~$i \leq m$''}
\smallskip

We can therefore abstract the iteration forcing and ask the instance~$I$ to satisfy the following property:

\smallskip
{\itshape For every $\Sigma^{0,I}_1$ predicate~$\varphi(U)$,
if for every~$x \in \omega$, there exists a finite $\Qsf$-instance~$A > x$
such that~$\varphi(A)$ is satisfied, then one of the~$A$'s is a subinstance of~$I$.}
\smallskip

Let us illustrate how this simplification works by reproving the 
separation of the Erd\H{o}s-Moser theorem from the ascending
descending sequence principle.

\begin{definition}[Ascending descending sequence]
A linear order is \emph{stable} if it is of order type~$\omega + \omega^{*}$.
$\ads$ is the statement ``Every linear order admits an infinite ascending or descending sequence''.
$\sads$ is the restriction of $\ads$ to stable linear orders.
\end{definition}

The ascending descending sequence principle has been studied
within the framework of reverse mathematics by Hirschfeldt \& Shore~\cite{Hirschfeldt2007Combinatorial}.
Lerman, Solomon \& Towsner~\cite{Lerman2013Separating} constructed
an infinite stable linear order~$I$ with $\omega$ and~$\omega^{*}$ parts
respectively~$B_0$ and~$B_1$, such that for every condition~$c$
and every~$\Sigma^{0,I}_1$ predicate $\Kcal^I$,
if for every~$x \in \omega$, there exists a finite set~$A > x$ and a finite
set of candidate extensions~$d_0, \dots, d_m$ of~$c$ such that~$\Kcal^I(A, d_i)$ is satisfied for each~$i \leq m$,
then one of the~$A$'s will be included in~$B_0$ and another one will be included in~$B_1$.
In particular, taking~$\Kcal^I(A, c) = \Phi_e^{c \oplus I} \cap A \neq \emptyset$, no infinite
solution to the constructed tournament $I$-computes a solution to~$I$.
After abstraction, we obtain the following property:

\smallskip
{\itshape For every $\Sigma^{0,I}_1$ predicate~$\varphi(U)$,
if for every~$x \in \omega$, there exists a finite set~$A > x$
such that~$\varphi(A)$ is satisfied, one of the~$A$'s is included in~$B_0$ and one of the~$A$'s is included in~$B_1$.}
\smallskip

Following the terminology of~\cite{Lerman2013Separating},
we say that a formula~$\varphi(U)$ is \emph{essential}
if for every $x \in \omega$, there exists some finite set~$A > x$
such that~$\varphi(A)$ holds.
This fairness property coincides with the notion of hyperimmunity for~$\overline{B_0}$ and~$\overline{B_1}$.

\begin{definition}[Preservation of hyperimmunity]\ 
\begin{itemize}
	\item[1.]
Let~$D_0, D_1, \dots$ be a computable list of all finite sets and let~$f$ be computable.
A \emph{c.e.\ array}~$\{ D_{f(i)}\}_{i \geq 0}$ is a c.e.\ set of mutually disjoint finite sets~$D_{f(i)}$.
A set~$B$ is \emph{hyperimmune} if for every c.e.\ array~$\{ D_{f(i)}\}_{i \geq 0}$, $D_{f(i)} \cap B = \emptyset$ for some~$i$.

	\item[2.] A $\Pi^1_2$ statement~$\Psf$ \emph{admits preservation of hyperimmunity}
if for each set~$Z$, each countable collection of $Z$-hyperimmune sets~$A_0, A_1, \dots$, 
	and each $\Psf$-instance $X \leq_T Z$
	there exists a solution $Y$ to~$X$ such that the $A$'s are $Y \oplus Z$-hyperimmune.
\end{itemize}
\end{definition}

The following lemma establishes the link between the fairness property
for~$\sads$ and the notion of hyperimmunity.

\begin{lemma}\label{lem:hyperimmune-fairness}
Fix a set~$Z$. A set~$B$ is $Z$-hyperimmune if and only if for every essential $\Sigma^{0,Z}_1$ predicate~$\varphi(U)$,
$\varphi(A)$ holds for some finite set~$A \subseteq \overline{B}$.
\end{lemma}

Hirschfeldt, Shore \& Slaman constructed in~\cite[Theorem 4.1]{Hirschfeldt2009atomic}
a stable computable linear order such that both the $\omega$ and the~$\omega^{*}$
part are hyperimmune. As every ascending (resp. descending) sequence
is an infinite subset of the $\omega$ (resp. $\omega^{*}$) part of the linear order,
we deduce the following theorem.

\begin{theorem}\label{thm:sads-hyperimmunity}
$\sads$ does not admit preservation of hyperimmunity.
\end{theorem}

A slight modification of the forcing in~\cite{Lerman2013Separating}
gives preservation of hyperimmunity of the Erd\H{o}s-Moser theorem.
We will however reprove it in a later section with a simpler forcing notion.
As expected, the notion of preservation of hyperimmunity
can be used to separate statements in reverse mathematics.

\begin{lemma}\label{lem:hyperimmunity-separation}
Fix two $\Pi^1_2$ statements~$\Psf$ and~$\Qsf$.
If $\Psf$ admits preservation of hyperimmunity and~$\Qsf$ does not,
then $\Psf$ does not imply~$\Qsf$ over~$\rca$.
\end{lemma}

Before starting an analysis of preservations of hyperimmunity 
for basic statements, we state another negative preservation result which enables
to reprove that the Erd\H{o}s-Moser theorem does not imply
the stable thin set theorem for pairs~\cite{Patey2013note}.

\begin{definition}[Thin set theorem]
Let $k \in \omega$ and $f : [\omega]^k \to \omega$.
A set $A$ is \textit{$f$-thin} if $f([A]^n) \neq \omega$, that is, if the set $A$ ``avoids''
at least one color.
$\ts^k$ is the statement ``every function $f : [\omega]^k \to \omega$ has
an infinite $f$-thin set''.
$\sts^2$ is the restriction of $\ts^2$ to stable functions.
\end{definition}

Introduced by Friedman in~\cite{FriedmFom:53:free}, the basic
reverse mathematics of the thin set theorem has been settled
by Cholak, Hirst \& Jockusch in~\cite{Cholak2001Free}.
Its study has been continued by Wang~\cite{Wang2014Some}, 
Rice~\cite{RiceThin}
and the author~\cite{PateyCombinatorial,Patey2015weakness}.
The author constructed in~\cite{Patey2014Somewhere}
an infinite computable stable function~$f : [\omega]^2 \to \omega$
such that the sets~$B_i = \{ n \in \omega : \lim_s f(n,s) \neq i\}$
are all hyperimmune. Every infinite $f$-thin set being an infinite
subset of one of the~$B$'s, we deduce the following theorem.

\begin{theorem}
$\sts^2$ does not admit preservation of hyperimmunity.
\end{theorem}

%% file: parts/part3-basic-hyperimmunity.tex
When defining a notion, it is usually convenient to see how it relates
with typical sets. There are two kinds of typicalities: genericity and randomness.
Both notions admit preservation of hyperimmunity.

\begin{theorem}\label{thm:typicality-hyperimmunity}
Fix some set~$Z$ and a countable collection of $Z$-hyperimmune sets~$B_0, B_1, \dots$
\begin{itemize}
	\item[1.] If~$G$ is sufficiently Cohen generic relative to~$Z$, the~$B$'s are $G \oplus Z$-hyperimmune.
	\item[2.] If~$R$ is sufficiently random relative to~$Z$, the~$B$'s are $R \oplus Z$-hyperimmune.
\end{itemize}
\end{theorem}

Note that this does not mean that the sets~$G$ and~$R$ are hyperimmune-free relative to~$Z$.
In fact, the converse holds: if $G$ is sufficiently generic and~$R$ sufficiently random, then both are $Z$-hyperimmune.
Some statements like the atomic model theorem ($\amt$), $\Pi^0_1$-genericity ($\pizog$)
and the rainbow Ramsey theorem for pairs~($\rrt^2_2$) are direct consequences
of genericity and randomness~\cite{Hirschfeldt2009atomic,Csima2009strength}. 
We can deduce from Theorem~\ref{thm:typicality-hyperimmunity} that they all
admit preservation of hyperimmunity.

Cohesiveness is a very useful statement in the analysis of Ramsey-type theorems
as it enables to transform an arbitrary instance into a stable one~\cite{Cholak2001strength}.
A set~$C$ is \emph{cohesive} for a sequence of sets~$R_0, R_1, \dots$
if~$C \subseteq^{*} R_i$ or~$C \subseteq^{*} \overline{R_i}$ for each~$i$.

\begin{theorem}\label{thm:coh-fairness-preservation}
$\coh$ admits preservation of hyperimmunity.
\end{theorem}

The proof is done by the usual construction of a cohesive set
with Mathias forcing, combined with the following lemma.

\begin{lemma}\label{lem:coh-preservation-lemma}
For every set~$Z$, every $Z$-computable Mathias condition~$(F, X)$, 
every~$\Sigma^{0,Z}_1$ formula~$\varphi(G, U)$ and every $Z$-hyperimmune set~$B$,
there exists an extension~$(E, Y)$ such that $X =^{*} Y$
and either $\varphi(G, U)$ is not essential for every set~$G$ satisfying~$(E,Y)$,
or $\varphi(E, A)$ holds for some finite set~$A \subseteq \overline{B}$.
\end{lemma}
\begin{proof}
Define
\[
\psi(U) = (\exists G \supseteq F)[G \subseteq F \cup X \wedge \varphi(G, U)]
\]
The formula~$\psi(U)$ is~$\Sigma^{0,Z}_1$. By hyperimmunity of~$B$,
either $\psi(U)$ is not essential, or~$\psi(A)$ holds for some finite set~$A \subseteq \overline{B}$.
In the first case, the condition~$(F,X)$ already satisfies the desired property.
In the second case, let~$A \subseteqfin \overline{B}$ be such that~$\psi(A)$ holds.
By the use property, there exists a finite set~$E$ satisfying~$(F, X)$ such that~$\varphi(E, A)$ holds.
Let $Y = X \setminus [0, max(E)]$. The condition $(E, Y)$ is a valid extension.\qed
\end{proof}

Weak König's lemma ($\wkl$) states that every infinite, binary tree
admits an infinite path. 

\begin{theorem}\label{thm:wkl-hyperimmunity}
$\wkl$ admits preservation of hyperimmunity.
\end{theorem}

Wei Wang~[personal communication] observed that $\wkl$ preserves hyperimmunity
in a much stronger sense than~$\coh$, since cohesive sets are of hyperimmune degree~\cite{Jockusch1993cohesive},
whereas by the hyperimmune-free basis theorem~\cite{Jockusch197201}, $\wkl$ can preserve
hyperimmunities of \emph{every} hyperimmune set simultaneously and not only countably many.

%% file: parts/part4-em-hyperimmunity.tex
The Erd\H{o}s-Moser theorem is a statement from graph theory 
which received a particular interest from reverse mathematical community
as it provides, together with the ascending descending sequence principle,
an alternative proof of Ramsey's theorem for pairs.

\begin{definition}[Erd\H{o}s-Moser theorem] 
A tournament $T$ is an irreflexive binary relation such that for all $x,y \in \omega$ with $x \not= y$, exactly one of $T(x,y)$ or $T(y,x)$ holds. A tournament $T$ is \emph{transitive} if the corresponding relation~$T$ is transitive in the usual sense. 
$\emo$ is the statement ``Every infinite tournament $T$ has an infinite transitive subtournament.''
\end{definition}

The Erd\H{o}s-Moser theorem was introduced in reverse mathematics by Bovykin \& Weiermann~\cite{Bovykin2005strength}
and then studied by Lerman, Solomon \& Towsner~\cite{Lerman2013Separating} and the author~\cite{Patey2015Degrees,PateyCombinatorial,Patey2014Controlling}.
In this section, we give a simple proof of the following theorem.

\begin{theorem}\label{thm:em-hyperimmunity}
$\emo$ admits preservation of hyperimmunity.
\end{theorem}

The proof of Theorem~\ref{thm:em-hyperimmunity} exploits the modularity of the framework by using
preservation of hyperimmunity of~$\wkl$. Together with previous preservations results, this theorem
is sufficient to reprove Theorem~\ref{thm:wang-theorem}.
We must first introduce some terminology.

\begin{definition}[Minimal interval]
Let $T$ be an infinite tournament and $a, b \in T$
be such that $T(a,b)$ holds. The \emph{interval} $(a,b)$ is the
set of all $x \in T$ such that $T(a,x)$ and $T(x,b)$ hold.
Let $F \subseteq T$ be a finite transitive subtournament of $T$.
For $a, b \in F$ such that $T(a,b)$ holds, we say that $(a,b)$
is a \emph{minimal interval of $F$} if there is no $c \in F \cap (a,b)$,
i.e. no $c \in F$ such that $T(a,c)$ and $T(c,b)$ both hold.
\end{definition}

\begin{definition}
An \emph{Erd\H{o}s Moser condition} (EM condition) for an infinite tournament~$T$
is a Mathias condition $(F, X)$ where
\begin{itemize}
	\item[(a)] $F \cup \{x\}$ is $T$-transitive for each $x \in X$
	\item[(b)] $X$ is included in a minimal $T$-interval of $F$.
\end{itemize}
\end{definition}

EM extension is Mathias extension. A set~$G$ \emph{satisfies} an EM condition~$(F, X)$
if it is $T$-transitive and satisfies the Mathias extension $(F, X)$.
Basic properties of EM conditions have been stated and proven in~\cite{Patey2015Degrees}.

Fix a set~$Z$ and some countable collection of $Z$-hyperimmune sets~$B_0, B_1, \dots$
Our forcing notion is the partial order of Erd\H{o}s Moser conditions~$(F,X)$
such that the~$B$'s are $X \oplus Z$-hyperimmune. Our initial condition is~$(\emptyset, \omega)$.
By Lemma~5.9 in~\cite{Patey2015Degrees}, EM conditions are extendable, 
so we can force the transitive subtournament to be infinite.
Therefore it suffices to prove the following lemma to deduce Theorem~\ref{thm:em-hyperimmunity}.

\begin{lemma}
Fix a condition~$(F,X)$, some $i \in \omega$ and some~$\Sigma^{0,Z}_1$
formula~$\varphi(G, U)$.
There exists an extension~$(E,Y)$ such that either $\varphi(G, U)$ is not essential
for every set~$G$ satisfying~$(E,Y)$,
or $\varphi(E, A)$ holds for some finite set~$A \subseteq \overline{B_i}$.
\end{lemma}
\begin{proof}
Let~$\psi(U)$ be the formula
``For every partition~$X_0 \cup X_1 = X$, there exists some~$j < 2$,
a $T$-transitive set~$G \subseteq X_j$ and a set~$\tilde{A} \subseteq U$
such that~$\varphi(F \cup G, \tilde{A})$ holds.''
By compactness, $\psi(U)$ is a~$\Sigma^{0,X \oplus Z}_1$ formula. By $X \oplus Z$-hyperimmunity of~$B_i$,
we have two cases:
\begin{itemize}
	\item Case 1: $\psi(A)$ holds for some finite set~$A \subseteq \overline{B_i}$.
	By compactness, there exists a finite set~$H \subset X$ such that for every partition~$H_0 \cup H_1 = H$,
	there exists some~$j < 2$, a $T$-transitive set~$G \subseteq H_j$ and a set~$\tilde{A} \subseteq A$
	such that~$\varphi(F \cup G, \tilde{A})$ holds. Given two sets~$U$ and~$V$, 
	we denote by~$U \to_T V$ the formula~$(\forall x \in U)(\forall y \in V)T(x,y)$.
	Each element~$y \in X$ induces a partition~$H_0 \cup H_1 = H$ such that~$H_0 \to_T \{y\} \to_T H_1$.
	There exists finitely many such partitions, so by the infinite pigeonhole principle,
	there exists an~$X$-computable infinite set~$Y \subset X$ and a partition~$H_0 \cup H_1 = H$
	such that~$H_0 \to_T Y \to_T H_1$. Let~$j < 2$ and~$G \subseteq H_j$ be the $T$-transitive set
	such that $\varphi(F \cup G, \tilde{A})$ holds for some~$\tilde{A} \subseteq A \subseteq \overline{B_i}$.
	By Lemma~5.9 in~\cite{Patey2015Degrees}, $(F \cup G, Y)$ is a valid extension. 

	\item Case 2: $\psi(U)$ is not essential with some witness~$x$.
	Then the~$\Pi^{0,X \oplus Z}_1$ class~$\Ccal$ of sets $X_0 \oplus X_1$ such that~$X_0 \cup X_1 = X$ and
	for every~$j < 2$, every $T$-transitive set~$G \subseteq X_j$ and every finite set~$\tilde{A} > x$,
	the formula $\varphi(F \cup G, \tilde{A})$ does not hold is not empty.
	By preservation of hyperimmunity of~$\wkl$, there exists some partition~$X_0 \oplus X_1 \in \Ccal$
	such that the~$B$'s are~$X_0 \oplus X_1 \oplus Z$-hyperimmune. The set~$X_j$
	is infinite for some~$j < 2$ and the condition~$(F, X_i)$ is the desired EM extension.
	\qed
\end{itemize}
\end{proof}

%% file: parts/part5-ts2-hyperimmunity.tex
There exists a fundamental difference in the way
$\sads$ and~$\sts^2$ witness their failure of preservation of hyperimmunity.
In the case of~$\sads$, we construct two hyperimmune sets 
whereas in the case of~$\sts^2$, a countable collection of hyperimmune sets is used.
This difference can be exploited to obtain further separation results.

\begin{definition}[Preservation of $n$ hyperimmunities]
A $\Pi^1_2$ statement~$\Psf$ \emph{admits preservation of $n$ hyperimmunities} 
if for each set~$Z$, each $Z$-hyperimmune sets~$A_0, \dots, \allowbreak A_{n-1}$,
and each $\Psf$-instance $X \leq_T Z$
there exists a solution $Y$ to~$X$ such that the~$A$'s are $Y \oplus Z$-hyperimmune.
\end{definition}

Theorem~\ref{thm:sads-hyperimmunity} shows that~$\sads$ does not admit preservation of 2 hyperimmunities.
On the other side, we shall see that~$\sts^2$ admits preservation of~$n$ hyperimmunities for every~$n \in \omega$.
Consider the following variants of the thin set theorem.

\begin{definition}[Thin set theorem]
Given a function $f : [\omega]^k \to n$, an infinite set~$H$
is \emph{$f$-thin} if~$|f([H]^k)| \leq n-1$ (i.e. $f$ avoids one color over~$H$).
For every~$k \geq 1$ and~$n \geq 2$, $\ts^k_n$ is the statement
``Every function~$f : [\omega]^k \to n$ has an infinite $f$-thin set''.
$\sts^2_n$ is the restriction of~$\ts^2_n$ to stable colorings.
\end{definition}

Note that~$\ts^2_2$ is Ramsey's theorem for pairs.
The following theorem is sufficient to separate $\ts^2$ from Ramsey's theorem for pairs
as $\ts^2 \leq_c \ts^2_n$ for every~$n \geq 2$.
The proof of preservation is rather technical and is therefore proven in appendix.

\begin{theorem}\label{thm:ts2-hyperimmunity-preservation}
For every~$n \geq 1$, $\ts^2_{n+1}$ admits preservation of $n$ but not $n+1$ hyperimmunities.
\end{theorem}

In the case~$n = 1$, noticing that the arithmetical comprehension scheme
($\aca$) does not preserve 1 hyperimmunities as witnessed by taking any $\Delta^0_2$ hyperimmune set,
we re-obtain the separation of Ramsey's theorem for pairs from $\aca$.
Hirschfeldt \& Jockusch~\cite{Hirschfeldtnotions} asked whether~$\ts^2_{n+1}$ implies $\ts^2_n$ over~$\rca$.
The author answered negatively in~\cite{Patey2015weakness}.
Preservation of $n$ hyperimmunities gives the same separation.

\begin{theorem}[Patey~\cite{Patey2015weakness}]
For every~$n \geq 2$,
let~$\Phi$ be the conjunction of~$\coh$, $\wkl$, $\rrt^2_2$,
$\pizog$, $\emo$, $\ts^2_{n+1}$.
Over~$\rca$, $\Phi$ does not imply any of~$\sads$ and $\sts^2_n$.
\end{theorem}

%% file: parts/part6-appendix.tex
\section{Preservation of hyperimmunity}

\begin{proof}[Lemma~\ref{lem:hyperimmune-fairness}]
Let~$D_0, D_1, \dots$ be a computable list of all finite sets.
\begin{itemize}
	\item
Fix some set~$Z$ and some $Z$-hyperimmune set~$B$.
For every essential $\Sigma^{0,Z}_1$ formula~$\varphi(U)$,
define the $Z$-computable function~$f$ inductively so that
$\varphi(D_{f(0)})$ holds and for every~$i$,
$D_{f(i+1)} > D_{f(i)}$ and~$\varphi(D_{f(i+1)})$ holds.
Because~$\varphi(U)$ is essential, the function~$f$ is total.
$\{ D_{f(i)}\}_{i \geq 0}$ is a $Z$-c.e.\ weak array, so by~$Z$-hyperimmunity,
$D_{f(i)} \cap B = \emptyset$ for some~$i$, hence $D_{f(i)} \subseteq \overline{B}$
and~$\varphi(D_{f(i)})$ holds.

	\item
Fix some set~$Z$ and some set~$B$ such that the fairness property of Lemma~\ref{lem:hyperimmune-fairness} holds.
For every $Z$-c.e.\ weak array $\{ D_{f(i)}\}_{i \geq 0}$, define the~$\Sigma^{0,Z}_1$ formula
$\varphi(U) = (\exists i)[U = D_{f(i)}]$. The formula~$\varphi(U)$ is essential, so
there exists some finite set~$A \subseteq \overline{B}$ such that~$\varphi(A)$ holds.
In particular, there exists some~$i$ such that~$D_{f(i)} \subseteq \overline{B}$. \qed
\end{itemize}
\end{proof}

\begin{proof}[Lemma~\ref{lem:hyperimmunity-separation}]
Fix a set~$X_0$, a countable collection of~$X_0$-hyperimmune sets~$B_0, B_1, \dots$
and an~$X_0$-computable $\Qsf$-instance~$J$ such that for every solution~$Y$ to~$J$,
one of the~$B$'s is not~$Y \oplus X_0$-hyperimmune.
By preservation of hyperimmunity of~$\Psf$ and carefully choosing
a sequence of $\Psf$-instances~$I_0, I_1, \dots$, we can define an infinite sequence
of sets~$X_1, X_2, \dots$ such that for each~$n \in \omega$
\begin{itemize}
	\item[(a)] $X_{n+1}$ is a solution to the~$\Psf$-instance~$I_n^{X_0 \oplus \dots \oplus X_n}$
	\item[(b)] The~$B$'s are $X_0 \oplus \dots \oplus X_n$-hyperimmune
	\item[(c)] For every~$X_0 \oplus \dots \oplus X_n$-computable $\Psf$-instance~$I$,
	there exists some~$m$ such that~$I = I_m^{X_0 \oplus \dots \oplus X_m}$.
\end{itemize}
Let~$\Mcal$ be the~$\omega$-structure whose second order part is
the Turing ideal
\[
\Ical = \{ Y : (\exists n)[ Y \leq_T X_0 \oplus \dots \oplus X_n] \}
\]
In particular, the~$\Qsf$-instance $J$ is in $\Ical$, but he~$B$'s are $Y$-hyperimmune for every~$Y \in \Ical$,
so $J$ has no solution~$Y \in \Ical$ and $\Mcal \not \models \Qsf$. 
By construction of~$\Ical$, every $\Psf$-instance~$I \in \Ical$
has a solution~$X_n \in \Ical$, so by Friedman~\cite{Friedman1974Some}, $\Mcal \models \rca \wedge \Psf$. \qed
\end{proof}

\section{Basic preservation proofs}

The proofs in this appendix are inspired by Wang's work in~\cite{Wang2014Definability}.

\begin{proof}[Item 1 of Theorem~\ref{thm:typicality-hyperimmunity}]
It suffices to prove that for every~$\Sigma^{0,Z}_1$ formula~$\varphi(G, U)$
and every~$i \in \omega$, the set of conditions~$\sigma$ forcing~$\varphi(G, U)$ not to be essential
or such that $\varphi(\sigma, A)$ holds for some finite set~$A \subset \overline{B_i}$ is dense.
Fix any string~$\sigma \in 2^{<\omega}$. Define
\[
\psi(U) = (\exists \tau \succeq \sigma)\varphi(\tau, U)
\]
The formula~$\psi(U)$ is~$\Sigma^{0,Z}_1$, so by $Z$-hyperimmunity of~$B_i$,
either~$\psi(U)$ is not essential, or~$\psi(A)$ holds for some finite set~$A \subseteq \overline{B_i}$.
If~$\psi(U)$ is not essential with witness $x \in \omega$,
then~$\sigma$ forces~$\varphi(G, U)$ not to be essential with the same witness.
If~$\psi(U)$ is essential, then there exists some finite set~$A \subset \overline{B_i}$ such that~$\psi(A)$ holds.
Unfolding the definition of~$\psi(A)$, there exists some~$\tau \succeq \sigma$ such that~$\varphi(\tau, A)$ holds.
The condition $\tau$ is an extension such that~$\varphi(\tau, A)$
holds for some~$A \subset \overline{B_i}$. \qed
\end{proof}

\begin{proof}[Item 2 of Theorem~\ref{thm:typicality-hyperimmunity}]
It suffices to prove that for every~$\Sigma^{0,Z}_1$ formula~$\varphi(G, U)$
and every~$i \in \omega$, the following class is Lebesgue null.
\[
\Scal = \{X : [\varphi(X,U) \mbox{ is essential }] \wedge (\forall A \subseteqfin \omega)\varphi(X,A) \imp A \not \subseteq \overline{B_i} \}
\]
Suppose it is not the case. There exists~$\sigma \in 2^{<\omega}$ such that
\[
\mu(X \in \Scal : \sigma \prec X) > 2^{-|\sigma|-1}
\]
Define
\[
\psi(U) = [\mu(X : (\exists \tilde{A} \subseteq U)\varphi(X,\tilde{A})) > 2^{-|\sigma|-1}]
\]
The formula~$\psi(U)$ is~$\Sigma^{0,Z}_1$ and by compactness, $\psi(U)$ is essential.
By~$Z$-hyperimmunity of~$B_i$, there exists some finite set~$A \subseteq \overline{B_i}$
such that~$\psi(A)$ holds. For every set~$A$ such that~$\psi(A)$ holds,
there exists some~$X \in \Scal$ and some~$\tilde{A} \subseteq A$
such that~$\varphi(X,\tilde{A})$ holds. By definition of~$X \in \Scal$, $\tilde{A} \not \subseteq \overline{B_i}$
and therefore~$A \not \subseteq \overline{B_i}$. Contradiction. \qed
\end{proof}

\begin{proof}[Theorem~\ref{thm:wkl-hyperimmunity}]
Fix some set~$Z$, some countable collection of~$Z$-hyperimmune sets~$B_0, B_1, \dots$
and some~$Z$-computable tree~$T \subseteq 2^{<\omega}$.
Our forcing conditions are~$(\sigma, R)$
where~$\sigma$ is a stem of the infinite, $Z$-computable tree~$R \subseteq T$.
A condition~$(\tau, S)$ \emph{extends} $(\sigma, R)$ if~$\sigma \preceq \tau$
and~$S \subseteq R$. The result is a direct consequence of the following lemma.

\begin{lemma}\label{lem:wkl-preservation-lemma}
For every condition~$c = (\sigma, R)$, every~$\Sigma^{0,Z}_1$ formula~$\varphi(G, U)$
and every~$i \in \omega$, there exists an extension~$d = (\tau, S)$ such that
$\varphi(P, U)$ is not essential for every path~$P \in [S]$, or $\varphi(\tau, A)$ holds
for some~$A \subseteq \overline{B_i}$.
\end{lemma}
\begin{proof}
Define
\[
\psi(U) = (\exists s)(\forall \tau \in R \cap 2^s)(\exists \tilde{A} \subseteqfin U)\varphi(\tau, \tilde{A})
\]
The formula~$\psi(U)$ is~$\Sigma^{0,Z}_1$ so we have two cases:
\begin{itemize}
	\item Case 1: $\psi(U)$ is not essential with some witness~$x$. By compactness,
	the following set is an infinite $Z$-computable subtree of~$R$:
	\[
	S = \{ \tau \in R : (\forall A > x) \neg \varphi(\tau, A) \}
	\]
	The condition~$d = (\sigma, S)$ is an extension such that~$\varphi(P, U)$
	is not essential for every~$P \in [S]$.
	\item Case 2: $\psi(U)$ is essential. By~$Z$-hyperimmunity of~$B_i$,
	there exists some finite set~$A \subseteq \overline{B_i}$ such that~$\psi(A)$ holds.
	Unfolding the definition of~$\psi(A)$, there exists some~$\tau \in R$ such that~$R^{[\tau]}$ is infinite
	and~$\varphi(\tau, \tilde{A})$ holds for some~$\tilde{A} \subseteq A \subseteq \overline{B_i}$. The condition
	$d = (\tau, R^{[\tau]})$ is an extension such that~$\varphi(\tau, \tilde{A})$
	holds for some finite set~$\tilde{A} \subseteq \overline{B_i}$.
\end{itemize}
\end{proof}

Using Lemma~\ref{lem:wkl-preservation-lemma}, define an infinite descending sequence 
of conditions~$c_0 = (\epsilon, T) \geq c_1 \geq \dots$
such that for each~$s \in \omega$
\begin{itemize}
	\item[(i)] $|\sigma_s| \geq s$
	\item[(ii)] $\varphi(P, U)$ is not essential for every path~$P \in [R_{s+1}]$, 
	or $\varphi(\sigma_{s+1}, A)$ holds
	for some finite set~$A \subseteq \overline{B_i}$ if~$s = \tuple{\varphi, i}$
\end{itemize}
where~$c_s = (\sigma_s, R_s)$. \qed
\end{proof}

\section{Thin set for pairs and preservation of hyperimmunity}

All the proofs in this section are very similar to~\cite{Patey2015weakness}.
We reprove everything in the context of preservation of hyperimmunities
for the sake of completeness.

\begin{definition}[Strong preservation of $n$ hyperimmunities]\ 
A~$\Pi^1_2$ statement~$\Psf$ \emph{admits strong preservation of $n$ hyperimmunities} 
if for each set $Z$, each $Z$-hyperimmune sets $B_0, \dots, B_{n-1}$ and each (arbitrary) $\Psf$-instance $X$, 
there exists a solution $Y$ to~$X$ such that the~$B$'s are $Y \oplus Z$-hyperimmune.
\end{definition}

The following lemma has been proven by the author in its full generality in~\cite{PateyCombinatorial}.
We reprove it in the context of preservation of $n$ hyperimmunities.

\begin{lemma}\label{lem:ts-strong-to-weak}
For every~$k,n \geq 1$ and~$m \geq 2$, if~$\ts^k_m$ admits strong preservation of~$n$ hyperimmunities,
then $\ts^{k+1}_m$ admits preservation of~$n$ hyperimmunities.
\end{lemma}
\begin{proof}
Fix any set~$Z$, some~$Z$-hyperimmune sets~$B_0, \dots, B_{n-1}$ and any $Z$-computable
coloring $f : [\omega]^{k+1} \to m$.
Consider the uniformly~$Z$-computable sequence of sets~$R_{\sigma,i}$ defined for each~$\sigma \in [\omega]^k$ and~$i < m$ by
\[
R_{\sigma,i} = \{s \in \omega : f(\sigma,s) = i\}
\]
As~$\coh$ admits preservation of~$n$ hyperimmunities, there exists
some~$\vec{R}$-cohesive set~$G$ such that the~$B$'s are $G \oplus Z$-hyperimmune.
The cohesive set induces a $(G \oplus Z)'$-computable coloring~$\tilde{f} : [\omega]^k \to m$ defined by:
\[
(\forall \sigma \in [\omega]^k) \tilde{f}(\sigma) = \lim_{s \in G} f(\sigma,s)
\]
As~$\ts^k_m$ admits strong preservation of $n$ hyperimmunities,
there exists an infinite $\tilde{f}$-thin set~$H$ such that
the~$B$'s are $H \oplus G \oplus Z$-hyperimmune.
$H \oplus G \oplus Z$ computes an infinite $f$-thin set.
\end{proof}

Thanks to Lemma~\ref{lem:ts-strong-to-weak}, it suffices
to prove the following theorem to deduce Theorem~\ref{thm:ts2-hyperimmunity-preservation}.

\begin{theorem}\label{thm:ts1-strong-preservation}
For every~$n \geq 1$, $\ts^1_{n+1}$ admits preservation of $n$ hyperimmunities.
\end{theorem}

The remainder of this section is devoted to the proof of Theorem~\ref{thm:ts1-strong-preservation}.
Fix some set~$Z$, some~$Z$-hyperimmune sets~$B_0, \dots, B_{n-1}$
and some $(n+1)$-partition~$A_0 \cup \dots \cup A_n = \omega$.
We will construct an infinite set~$G$ such that $G \cap \overline{A_i}$
is infinite for each~$i \leq n$ and the $B$'s are $(G \cap \overline{A_i}) \oplus Z$-hyperimmune
for some~$i \leq n$.
Our forcing conditions are Mathias conditions~$(F, X)$ 
such that the $B$'s are $X \oplus Z$-hyperimmune.

\subsection{Forcing limitlessness}

We want to satisfy the following scheme of requirements to ensure that~$G \cap \overline{A_i}$
is infinite for each~$i \leq n$.

\[
\Qcal_p : (\exists m_0, \dots, m_n > p)[m_0 \in G \cap \overline{A_0} \wedge \dots \wedge m_n \in G \cap \overline{A_n}]
\]

We say that an~$(n+1)$-partition~$A_0 \cup \dots \cup A_{n-1} = \omega$ is \emph{non-trivial} 
if there exists no infinite $f$-thin set~$H$ such that the $B$'s are $H \oplus Z$-hyperimmune.
A condition~$(F, X)$ \emph{forces $\Qcal_p$}
if there exists~$m_0, \dots, m_n > p$ such that~$m_i \in F \cap \overline{A_i}$
for each~$i \leq n$. Therefore, if~$G$ satisfies~$c$ and~$c$ forces~$\Qcal_p$,
then~$G$ satisfies the requirement~$\Qcal_p$.
We now prove that the set of conditions forcing~$\Qcal_p$ is dense for each~$p \in \omega$.
Therefore, every sufficiently generic filter will induce~$n+1$ infinite sets~$G \cap \overline{A_0}, \dots, G \cap \overline{A_n}$.

\begin{lemma}\label{lem:ts2-reduc-force-Q}
For every condition~$c$ and every~$p \in \omega$, there is a condition~$d$
extending~$c$ such that~$d$ forces~$\Qcal_p$.
\end{lemma}
\begin{proof}
Fix some~$p \in \omega$. It is sufficient to show that for a condition~
$c = (F, X)$ and some~$i \leq n$,
there exists an extension~$d_0 = (H, Y)$ and some integer~$m_i > p$
that~$m_i \in H \cap \overline{A_i}$.
By iterating the process for each~$i \leq n$, we obtain the desired extension~$d$.
Suppose for the sake of contradiction that~$X \cap \overline{A_i}$ is finite.
Then one can $X$-compute an infinite set~$H$ thin for the~$A$'s with witness $j$ for any~$j \neq i$, 
contradicting non-triviality of~$f$.
Therefore, there exists an~$m_i \in X \cap \overline{A_i}$, $m_i > p$.
The condition $d_0 = (F \cup \{m_i\}, X \setminus [0, m_i] )$
is the desired extension. \qed
\end{proof}

\subsection{Forcing non-preservation}

Fix an enumeration~$\varphi_0(G, U), \varphi_1(G, U), \dots$ of all~$\Sigma^{0,Z}_1$ formulas.
The second scheme of requirements consists in ensuring that
the sets~$B_0, \dots, B_{n-1}$ are all $G \cap \overline{A_i}$-hyperimmune for some~$i \leq n$.
The requirements are of the following form for each~$\vec{e}$.

\[
\Rcal_{\vec{e}} : \bigwedge_{j < n} \Rcal^{A_0, B_j}_{e_0}
	\vee \dots \vee \bigwedge_{j < n} \Rcal^{A_n, B_j}_{e_n}
\]
where
\[
\Rcal^{A_i, B_j}_e : \varphi_e(G \cap \overline{A_i}, U) \mbox{ essential } 
	\Rightarrow (\exists A \subseteqfin \overline{B_j}) \varphi_e(G \cap \overline{A_i}, A)
\]

A condition~\emph{forces $\Rcal_{\vec{e}}$} if every set~$G$ satisfying this condition also satisfies 
the requirement~$\Rcal_{\vec{e}}$.

\begin{lemma}\label{lem:ts2-reduc-step}
For every condition~$c = (F, X)$, every~$i_0 < i_1 \leq n$,
every $j < n$ and every indices~$\vec{e}$, there exists
an extension~$d$ such that for some~$i \in \{i_0, i_1\}$, $d$ forces~$\varphi_{e_i}(G \cap \overline{A_i}, U)$ not to be essential
or forces~$\varphi_{e_i}(G \cap \overline{A_i}, A)$ for some finite set~$A \subseteq \overline{B_j}$.
\end{lemma}
\begin{proof}
Let~$\psi(U)$ be the formula which holds if for every $2$-partition~$X_{i_0} \cup X_{i_1} = X$, there is some $i \in \{i_0,i_1\}$
and some set~$G_i \subseteq X_i$ such that~$\varphi_{e_i}((F \cap \overline{A_i}) \cup G_i, \tilde{A})$ 
holds for some~$\tilde{A} \subseteq U$. By compactness, the formula~$\psi(U)$ is~$\Sigma^{0,X \oplus Z}_1$. We have two cases:
\begin{itemize}
	\item Case 1:  $\psi(U)$ is essential. As~$B_j$ is $X \oplus Z$-hyperimmune, there exists some
	finite set~$A \subseteq \overline{B_j}$ such that~$\psi(A)$ holds.
	In particular, taking~$X_i = X \cap \overline{A_i}$ for each~$i \in \{i_0, i_1\}$,
	there exists some~$i \in \{i_0, i_1\}$ and some finite set~$G_i \subseteq X_i$
	such that~$\varphi_{e_i}((F \cap \overline{A_i}) \cup G_i, \tilde{A})$ holds for some~$\tilde{A} \subseteq A$.
	The condition~$d = (F \cup G_i, X \setminus [0, max(G_i)])$
	is an extension forcing~$\varphi_{e_i}(G \cap \overline{A_i}, A)$ for some finite set~$A \subseteq \overline{B_j}$

	\item Case 2: $\psi(U)$ is not essential, say with witness~$x$. By compactness, the~$\Pi^{0,X \oplus Z}_1$ class~$\Ccal$
	of sets~$X_{i_0} \oplus X_{i_1}$ such that $X_{i_0} \cup X_{i_1} = X$ and for every~$A > x$, every~$i \in \{i_0, i_1\}$
	and every set~$G_i \subseteq X_i$, $\varphi_{e_i}((F \cap \overline{A_i}) \cup G_i, A)$ does not hold
	is not empty. By preservation of hyperimmunity of~$\wkl$,
	there exists some~$X_{i_0} \oplus X_{i_1} \in \Ccal$ such that
	the~$B$'s are~$X_{i_0} \oplus X_{i_1} \oplus Z$-hyperimmune.
	Let~$i \in \{i_0, i_1\}$ be such that~$X_i$ is infinite.
	The condition~$d = (F, X_i)$ is an extension of~$c$ forcing~$\varphi_{e_i}(G \cap \overline{A_i}, U)$
	not to be essential. \qed
\end{itemize}
\end{proof}

\begin{lemma}\label{lem:ts2-reduc-force-R}
For every condition~$c$,
and every indices~$\vec{e}$, there exists
an extension~$d$ forcing~$\Rcal_{\vec{e}}$.
\end{lemma}
\begin{proof}
Fix a condition~$c$, and apply iteratively Lemma~\ref{lem:ts2-reduc-step}
to obtain an extension~$d$ such that for each~$j < n$, 
$d$ forces~$\varphi_{e_i}(G \cap \overline{A_i}, U)$ not to be essential or
forces $\varphi_{e_i}(G \cap \overline{A_i}, A)$ for some finite set~$A \subseteq \overline{B_j}$ for $n$ different~$i$'s.
By the pigeonhole principle, there exists some~$i \leq n$
such that $d$ forces~$\varphi_{e_i}(G \cap \overline{A_i}, U)$ not to be essential or
forces $\varphi_{e_i}(G \cap \overline{A_i}, A)$ for some finite set~$A \subseteq \overline{B_j}$ for each~$j < n$.
Therefore, $d$ forces~$\Rcal_{\vec{e}}$. \qed
\end{proof}

\subsection{Construction}

Thanks to Lemma~\ref{lem:ts2-reduc-force-R} and Lemma~\ref{lem:ts2-reduc-force-Q}, 
define an infinite descending sequence of conditions
$c_0 = (\emptyset, \omega) \geq c_1 \geq \dots$ such that for each~$s \in \omega$,
\begin{itemize}
	\item[(a)] $c_{s+1}$ forces~$\Rcal_{\vec{e}}$ if~$s = \tuple{\vec{e}}$
	\item[(b)] $c_{s+1}$ forces~$\Qcal_s$
\end{itemize}
where~$c_s = (F_s, X_s)$.
Let~$G = \bigcup_s F_s$. The sets~$G \cap \overline{A_0}, \dots, G \cap \overline{A_n}$ are all infinite
and the~$B$'s are $(G \cap \overline{A_i}) \oplus Z$-hyperimune for some~$i \leq n$.
This finishes the proof.